\documentclass{amsart}
\usepackage{amsmath,amscd,amsthm,amsfonts,amssymb}

\theoremstyle{plain}
\newtheorem{theorem}                 {Theorem}      [section]
\newtheorem{proposition}  [theorem]  {Proposition}

\theoremstyle{definition}
\newtheorem{example}      [theorem]  {Example}
\newtheorem{remark}       [theorem]  {Remark}
\newtheorem{definition}   [theorem]  {Definition}

\numberwithin{equation}{section}

\def \theo-intro#1#2 {\vskip .25cm\noindent{\bf Theorem #1\ }{\it #2}}

\def \rn{\mathbb R}

\def \F{\mathcal F}
\def \H{\mathcal H}

\def \V{\mathcal V}

\def \sol{\mathfrak{sol}}

\def \ip #1#2{\langle #1,#2 \rangle}

\def \lb#1#2{[#1,#2]}

\def \g{\mathfrak{g}}
\def \h{\mathfrak{h}}
\def \k{\mathfrak{k}}

\def \r{\mathfrak{r}}

\DeclareMathOperator{\ad}{ad}

\def \SLR#1{\text{\bf SL}_{#1}(\rn)}
\def \slr#1{\mathfrak{sl}_{#1}(\rn)}

\def \SO#1{\text{\bf SO}(#1)}
\def \so#1{\mathfrak{so}(#1)}

\def \SU#1{\text{\bf SU}(#1)}
\def \su#1{\mathfrak{su}(#1)}

\def \co#1{\mathfrak{co}(#1)}

\def \nab#1#2{\hbox{$\nabla$\kern -.3em\lower 1.0 ex
    \hbox{$#1$}\kern -.1 em {$#2$}}}

\begin{document}
\baselineskip 22pt \larger

\allowdisplaybreaks

\title{Harmonic morphisms from\\ five-dimensional Lie groups}

\author{Sigmundur Gudmundsson}



\address
{Department of Mathematics, Faculty of Science, Lund University,
Box 118, S-221 00 Lund, Sweden}
\email{Sigmundur.Gudmundsson@math.lu.se}

\begin{abstract}
We consider five-dimensional Lie groups equipped with a left-invariant
Riemannian metric. On such groups we construct left-invariant
conformal foliations with minimal leaves of codimension 2.
These foliations produce complex-valued harmonic morphisms
locally defined on the Lie group.
\end{abstract}

\maketitle

\section{Introduction}

Harmonic morphisms $\phi:(M,g)\to(N,h)$ between Riemannian manifolds are maps
which preserve Laplace's equation.  They are
solutions to over-determined non-linear systems of partial differential
equations determined by the geometric data of the manifolds involved.
For this reason, harmonic morphisms are difficult to find and have no
general existence theory, not even locally.  It is therefore important
to develop tools for manufacturing examples and thereby to prove existence
in special cases.

We are interested in complex-valued harmonic morphisms from Riemannian
homogeneous spaces.  It is well-known that any symmetric space,
which is neither $G_2/\SO 4$ nor its non-compact dual, carries local
harmonic morphisms and even global solutions exist if the
space is of non-compact type, see \cite{Gud-Sve-4}.  Here we focus
our attention on Riemannian Lie groups as homogeneous spaces.

In \cite{Gud-Sve-5}, \cite{Gud-Nor-1} and \cite{Gud-Sve-6} the authors
introduce a method for constructing left-invariant conformal foliations
with minimal leaves of codimension 2.  These produce local complex-valued
harmonic morphisms on the Lie groups involved.  This technique gives an
extensive collection of new solutions to our non-linear problem in
dimensions 3 and 4.  It even provides a classification of those Riemannian
Lie groups carrying such foliation in these cases.

In this paper, we consider 5-dimensional Lie groups $G$.  The constructed
foliations $\F$ are 3-dimensional with an integrable tangent distribution $\V$.
Since this is left-invariant it is induced by a 3-dimensional subalgebra $\k$
of the Lie algebra $\g$ of $G$. L. Bianchi classified the $3$-dimensional
real Lie algebras up to isomorphy.  They fall into nine disjoint types I-IX.
Bianchi's list can be found in Appendix \ref{section-Bianchi}.

\section{Harmonic morphisms and minimal conformal foliations}

Let $M$ and $N$ be two manifolds of dimensions $m$ and $n$,
respectively. A Riemannian metric $g$ on $M$ gives rise to the
notion of a {\it Laplacian} on $(M,g)$ and real-valued {\it
harmonic functions} $f:(M,g)\to\rn$. This can be generalized to
the concept of {\it harmonic maps} $\phi:(M,g)\to (N,h)$ between
Riemannian manifolds, which are solutions to a semi-linear system
of partial differential equations, see \cite{Bai-Woo-book}.

\begin{definition}
  A map $\phi:(M,g)\to (N,h)$ between Riemannian manifolds is
  called a {\it harmonic morphism} if, for any harmonic function
  $f:U\to\rn$ defined on an open subset $U$ of $N$ with $\phi^{-1}(U)$
non-empty,
  $f\circ\phi:\phi^{-1}(U)\to\rn$ is a harmonic function.
\end{definition}

The following characterization of harmonic morphisms between
Riemannian manifolds is due to Fuglede and T. Ishihara.  For the
definition of horizontal (weak) conformality (alt. semiconformality)
we refer to \cite{Bai-Woo-book}.

\begin{theorem}\cite{Fug-1,T-Ish}
  A map $\phi:(M,g)\to (N,h)$ between Riemannian manifolds is a
  harmonic morphism if and only if it is a horizontally (weakly)
  conformal harmonic map.
\end{theorem}

Let $(M,g)$ be a Riemannian manifold, $\V$ be an integrable
distribution on $M$ and denote by $\H$ its orthogonal
complement distribution on $M$.
As customary, we also use $\V$ and $\H$ to denote the
orthogonal projections onto the corresponding subbundles of $TM$
and denote by $\F$ the foliation tangent to
$\V$. The second fundamental form for $\V$ is given by
$$B^\V(U,V)=\frac 12\H(\nabla_UV+\nabla_VU)\qquad(U,V\in\V),$$
while the second fundamental form for $\H$ is given by
$$B^\H(X,Y)=\frac{1}{2}\V(\nabla_XY+\nabla_YX)\qquad(X,Y\in\H).$$
The foliation $\F$ tangent to $\V$ is said to be {\it conformal} if there is a
vector field $V\in \V$ such that $$B^\H=g\otimes V,$$ and
$\F$ is said to be {\it Riemannian} if $V=0$.
Furthermore, $\F$ is said to be {\it minimal} if $\text{trace}\ B^\V=0$ and
{\it totally geodesic} if $B^\V=0$. This is equivalent to the
leaves of $\F$ being minimal and totally geodesic submanifolds
of $M$, respectively.

It is easy to see that the fibres of a horizontally conformal
map (resp.\ Riemannian submersion) give rise to a conformal foliation
(resp.\ Riemannian foliation). Conversely, the leaves of any
conformal foliation (resp.\ Riemannian foliation) are
locally the fibres of a horizontally conformal map
(resp.\ Riemannian submersion), see \cite{Bai-Woo-book}.

The next result of Baird and Eells gives the theory of
harmonic morphisms, with values in a surface,
a strong geometric flavour.

\begin{theorem}\cite{Bai-Eel}\label{theo:B-E}
Let $\phi:(M^m,g)\to (N^2,h)$ be a horizontally conformal
submersion from a Riemannian manifold to a surface. Then $\phi$ is
harmonic if and only if $\phi$ has minimal fibres.
\end{theorem}

For the general theory of harmonic morphisms between Riemannian
manifolds we refer to the excellent book \cite{Bai-Woo-book}
and the regularly updated on-line bibliography \cite{Gud-bib}.

\section{Conformal foliations of codimension 2}
\label{Conformal foliations of codimension 2}

Let $G$ be a Riemannian Lie group, $K$ be a subgroup of codimension 2
and $\g$,  $\k$ be their Lie algebras, respectively.  Further assume that the
foliation generated be the integrable distribution $\V\cong\k$ is conformal.
This means that for any $X\in\V$ the adjoint action of $X$ on the
2-dimensional horizontal subspace $\H$ is conformal i.e. there exists
a $\rho\in\rn$ such that
$$\ip{\ad_XA}{B}+\ip{A}{\ad_XB}=\rho\cdot\ip{A}{B}\qquad(A,B\in\H),$$
or put differently, $\H\ad_X\big\vert_\H:\H\to\H$ is in the conformal algebra
$$\co{\H}=\rn\cdot\mathrm{Id}_\H\oplus\so{\H}$$ acting on $\H$.
It follows from the Jacobi identity and the fact that $\V$
is integrable that this is indeed a Lie algebra representation of
$\k$ on $\H$ i.e. $$\H\ad_{[X,Y]}\big\vert_\H
=[\H\ad_X\big\vert_\H,\H\ad_Y\big\vert_\H]\qquad(X,Y\in\V).$$
Here the bracket on the right-hand side is just the usual
bracket on the space of endomorphisms on $\H$. Since
$\co{\H}$ is abelian, we see from this formula,
that the adjoint action of $[\V,\V]$ has no $\H$-component.
These arguments were already utilized in the 4-dimensional case,
see \cite{Gud-Sve-6}. Example \ref{exam-counter}
shows that they are not valid when the codimension is greater that 2.

\begin{remark}
It should be noted that if the operator $\H\ad_x|_{\H}:\H\to\H$
is skew-symmetric for $X\in\V$ then the constant $\rho\in\rn$ is zero.
This applies, for example, in the important case when $G$ is a compact semisimple
Lie group and the metric is a negative multiple of the corresponding
biinvariant Killing form.
\end{remark}

\begin{example}\label{exam-counter}
Let $G$ be a 6-dimensional Riemannian Lie group such that
$$\{X,Y,Z,A,B,C\}$$ is an orthonormal basis for the Lie algebra
$\g$ with the following bracket relations
$$\lb XY =Z,\ \
\lb ZX =Y,\ \
\lb YZ =X,$$
$$\lb XA =-C-X-Z,\ \
\lb XB =-Y-Z,\ \
\lb XC =A+X+Y,$$
$$\lb YA =Z,\ \
\lb YB =C+2X+Z,\ \
\lb YC =-B-X-Y+2Z,\ \ $$
$$\lb ZA =-B-2Y+Z,\ \
\lb ZB =A+2X,\ \
\lb ZC =-Y,$$
$$\lb BA =-B+2C+2X-2Y+Z,\ \
\lb CA = A-B+C+X+2Z,$$
$$\lb CB = A-B+Y+2Z.$$
Then the corresponding foliation generated by $X,Y,Z\in\g$
is conformal with minimal leaves which are not totally geodesic.
The adjoint action of the derived algebra $[\su 2,\su 2]=\su 2$
on the horizontal distribution $\H$ has a non-trivial $\H$-component.
This example was obtained by assuming that $\k=\su 2$ and then
playing with the corresponding Jacobi conditions.
\end{example}

\begin{proposition}
Let $G$ be a Riemannian Lie group, $K$ be a subgroup of
codimension 2 and $\g$, $\k$ be their Lie algebras,
respectively.  Further let the foliation generated by the
integrable distribution $\V\cong\k$ be conformal.  If $K$ is
semisimple then the corresponding foliation is Riemannian.
\end{proposition}

\begin{proof}  The result is a direct consequence of $[\k,\k]=\k$
and the fact that the adjoint action of [V,V] has no H-component.
\end{proof}

From now on we shall assume that $G$ is a 5-dimensional Riemannian
Lie group $G$ and that $\{X,Y,Z,A,B\}$ is an orthonormal basis
for its Lie algebra $\g$.  Further let $X,Y,Z\in\g$ generate a
3-dimensional left-invariant integrable distribution $\V$ on $G$
which is conformal and with minimal leaves of codimension 2.
Let $\H=\V^\perp$ be the horizontal distribution generated by $A,B\in\g$.
Then it is easily seen that $A,B$ can be chosen in such a way
that the Lie bracket relations for $\g$ are of the following form
\begin{eqnarray*}
\lb XY &=& c_1X+c_2Y+c_3Z,\\
\lb ZX &=& c_4X+c_5Y+c_6Z,\\
\lb YZ &=& c_7X+c_8Y+c_9Z,
\end{eqnarray*}
\begin{eqnarray*}
\lb XA &=&  a_1 A+b_1 B+x_1 X+y_1 Y+z_1 Z,\\
\lb XB &=& -b_1 A+a_1 B+x_2 X+y_2 Y+z_2 Z,\\
\lb YA &=&  a_2 A+b_2 B+x_3 X+y_3 Y+z_3 Z,\\
\lb YB &=& -b_2 A+a_2 B+x_4 X+y_4 Y+z_4 Z,\\
\lb ZA &=&  a_3 A+b_3 B+x_5 X+y_5 Y-(x_1+y_3) Z,\\
\lb ZB &=& -b_3 A+a_3 B+x_6 X+y_6 Y-(x_2+y_4) Z,\\
\lb AB &=& r A+\theta_1 X+\theta_2 Y+\theta_3 Z.
\end{eqnarray*}
For later reference we state the following easy result describing
the geometry of the situation.

\begin{proposition}\label{prop-geometry}
Let $G$ be a 5-dimensional Lie group and $\{X,Y,Z,A,B\}$ be an
orthonormal basis for its Lie algebra $\g$ as above.  Then
\begin{enumerate}
\item[(i)] $\V$ is {\it Riemannian} if and only if $a=(a_1,a_2,a_3)=0$,
\item[(ii)] $\H$ is {\it integrable} if and only if
$\theta=(\theta_1,\theta_2,\theta_3)=0$, and
\item[(iii)] $\V$ is {\it totally geodesic} if and only if
$$\Lambda(x,y,z)=
\begin{pmatrix}
x_1 & y_3 & x_3+y_1 & x_5+z_1 & y_5+z_3\\
x_2 & y_4 & x_4+y_2 & x_6+z_2 & y_6+z_4
\end{pmatrix}=0.
$$
\end{enumerate}
\end{proposition}

\section{The Lie algebra $\su 2$ of type IX}

The special unitary group $\SU 2$ is diffeomorphic to the
standard three dimensional sphere $S^3$.  Its Lie algebra is generated by
$$
e_1=\begin{pmatrix}0 & -1\\ 1 &  0\end{pmatrix}, \ \
e_2=\begin{pmatrix}i &  0\\ 0 & -i\end{pmatrix}, \ \
e_3=\begin{pmatrix}0 &  i\\ i &  0\end{pmatrix}.
$$
As a compact Lie group $\SU 2$ carries, amongst others, the well-known
one dimensional family of biinvariant Berger metrics $g_\lambda$
with $\lambda\in\rn^+$.
In terms of the orthonormal basis $\{X=e_1/\lambda,Y=e_2,Z=e_3\}$,
with respect to $g_\lambda$, the Lie brackets are then given by
$$\lb XY=2\lambda Z,\ \ \lb ZX=2\lambda Y,\ \ \lb YZ=2X/\lambda.$$
We extend this to an orthonormal basis $\{X,Y,Z,A,B\}$ for the Lie
algebra $\g$ of a 5-dimensional Lie group $G$, carrying a left-invariant,
minimal and conformal integrable distribution generated by $\su 2$.
It follows from our discussion in Section
\ref{Conformal foliations of codimension 2}, that the adjoint action of
$\su 2=[\su 2,\su 2]$, on the horizontal $\H$, has no $\H$-component.
This means that the additional Lie bracket relations for $\g$ take the form
\begin{eqnarray*}
\lb XA &=& x_1 X+y_1 Y+z_1 Z,\\
\lb XB &=& x_2 X+y_2 Y+z_2 Z,\\
\lb YA &=& x_3 X+y_3 Y+z_3 Z,\\
\lb YB &=& x_4 X+y_4 Y+z_4 Z,\\
\lb ZA &=& x_5 X+y_5 Y-(x_1+y_3) Z,\\
\lb ZB &=& x_6 X+y_6 Y-(x_2+y_4) Z,\\
\lb AB &=& r A+\theta_1 X+\theta_2 Y+\theta_3 Z.
\end{eqnarray*}
The following result tells us that for each $\lambda\in\rn^+$ we obtain
a 7-dimensional family of solutions to our problem.

\begin{theorem}\label{theo-SU2}
Let $G$ be a 5-dimensional Riemannian Lie group carrying a left-invariant,
minimal and conformal distribution $\V$, generated by the Riemannian
subalgebra $\su 2$ of $\g$ as above.  Then the Lie bracket relations
are of the form
$$\lb XY=2\lambda Z,\ \ \lb ZX=2\lambda Y,\ \ \lb YZ=2X/\lambda.$$
\begin{eqnarray*}
\lb XA &=& -\lambda^2x_3 Y-\lambda^2x_5 Z,\\
\lb XB &=& -\lambda^2x_4 Y-\lambda^2x_6 Z,\\
\lb YA &=& x_3 X+z_3 Z,\\
\lb YB &=& x_4 X+z_4 Z,\\
\lb ZA &=& x_5 X-z_3 Y,\\
\lb ZB &=& x_6 X-z_4 Y,\\
\lb AB &=& r A+\theta_1 X+\theta_2 Y+\theta_3 Z,
\end{eqnarray*}
where $\theta_1,\theta_2,\theta_3$ are given by
\begin{equation*}
\begin{pmatrix}
\theta_1  \\
\theta_2 \\
\theta_3
\end{pmatrix}
=\frac 12
\begin{pmatrix}
rz_3/\lambda+\lambda(x_3x_6-x_4x_5) \\
\lambda(rx_5-x_3z_4+x_4z_3) \\
-\lambda(rx_3+z_4x_5-z_3x_6)
\end{pmatrix}.
\end{equation*}
The corresponding foliation $\F$ is Riemannian. It is totally geodesic
if and only if $\lambda=1$ i.e the leaves are 3-dimensional round spheres.
\end{theorem}

\begin{proof}
It is easily checked that the Jacobi relations for $\g$ are equivalent
to the following system of equations
$$\Omega(\lambda,x,y,z)=
\begin{pmatrix}
x_1 & y_3 & y_1+\lambda^2x_3 & z_1+\lambda^2x_5 & y_5+z_3\\
x_2 & y_4 & y_2+\lambda^2x_4 & z_2+\lambda^2x_6 & y_6+z_4
\end{pmatrix}=0.
$$
This implies that the Lie bracket relations take the stated form.
It follows from Proposition \ref{prop-geometry} that the foliation
$\F$ is Riemannian.  The same result shows that $\F$ is totally
geodesic if and only if $\lambda=1$.
\end{proof}

To the structure of the Lie algebra $\g$ in Theorem
\ref{theo-SU2} we can say the following. In the generic case,
we have a Levi decomposition $\g=\k\oplus\r$ with the radical $\r$
generated by
$$R_1=X+\frac{\lambda^2(x_3x_6-x_4x_5)}{x_5z_4-x_6z_3}Z
-\frac{2\lambda x_6}{x_5z_4-x_6z_3}A+\frac{2\lambda x_5}{x_5z_4-x_6z_3}B$$
and
$$R_2=Y-\frac{x_3z_4-x_4z_3}{x_5z_4-x_6z_3}Z
+\frac{2z_4}{\lambda(x_5z_4-x_6z_3)}A-\frac{2z_3}{\lambda(x_5z_4-x_6z_3)}B.$$
In the light of Proposition 5.1 of \cite{Gud-Sve-4}, it is
interesting to notice that in the generic case the dimension of the quotient
algebra $\g/[\g,\g]$ is clearly 1 if and only if $r\neq 0$.

\section{The Lie algebra $\slr 2$ of type VIII}

The special linear group $\SLR 2$ is the non-compact companion of
$\SU 2$.  Its Lie algebra $\slr 2$ is generated by
$$
e_1=\begin{pmatrix}0 & -1\\ 1 &  0\end{pmatrix}, \ \
e_2=\begin{pmatrix}1 &  0\\ 0 & -1\end{pmatrix}, \ \
e_3=\begin{pmatrix}0 &  1\\ 1 &  0\end{pmatrix}.
$$
As a Lie group it carries, amongst others, a one dimensional family
of left-invariant metrics $g_\lambda$ with $\lambda\in\rn^+$.
In terms of the orthonormal basis $$\{X=e_1/\lambda,Y=e_2,Z=e_3\}$$
with respect to $g_\lambda$ the Lie brackets are given by
$$\lb XY=2\lambda Z,\ \ \lb ZX=2\lambda Y,\ \ \lb YZ=-2X/\lambda.$$
We extend this to an orthonormal basis $\{X,Y,Z,A,B\}$ for the Lie
algebra $\g$ of a 5-dimensional Lie group $G$, carrying a left-invariant,
minimal and conformal integrable distribution generated by $\slr 2$.
It follows from our discussion in Section
\ref{Conformal foliations of codimension 2}, that the adjoint action of
$\slr 2=[\slr 2,\slr 2]$, on the horizontal $\H$, has no $\H$-component.
This means that the additional Lie bracket relations for $\g$ take the form
\begin{eqnarray*}
\lb XA &=& x_1 X+y_1 Y+z_1 Z,\\
\lb XB &=& x_2 X+y_2 Y+z_2 Z,\\
\lb YA &=& x_3 X+y_3 Y+z_3 Z,\\
\lb YB &=& x_4 X+y_4 Y+z_4 Z,\\
\lb ZA &=& x_5 X+y_5 Y-(x_1+y_3) Z,\\
\lb ZB &=& x_6 X+y_6 Y-(x_2+y_4) Z,\\
\lb AB &=& r A+\theta_1 X+\theta_2 Y+\theta_3 Z.
\end{eqnarray*}
As in the case of $\SU 2$, we get for each $\lambda\in\rn^+$
a 7-dimensional family of solutions.

\begin{theorem}\label{theo-SL2}
Let $G$ be a 5-dimensional Riemannian Lie group carrying a left-invariant,
minimal and conformal distribution $\V$, generated by the Riemannian
subalgebra $\slr 2$ of $\g$ as above.  Then the Lie bracket relations
are of the form
$$\lb XY=2\lambda Z,\ \ \lb ZX=2\lambda Y,\ \ \lb YZ=-2X/\lambda.$$
\begin{eqnarray*}
\lb XA &=& \lambda^2x_3 Y+\lambda^2x_5 Z,\\
\lb XB &=& \lambda^2x_4 Y+\lambda^2x_6 Z,\\
\lb YA &=& x_3 X+z_3 Z,\\
\lb YB &=& x_4 X+z_4 Z,\\
\lb ZA &=& x_5 X-z_3 Y,\\
\lb ZB &=& x_6 X-z_4 Y,\\
\lb AB &=& r A+\theta_1 X+\theta_2 Y+\theta_3 Z,
\end{eqnarray*}
where $\theta_1,\theta_2,\theta_3$ are given by
\begin{equation*}
\begin{pmatrix}
\theta_1  \\
\theta_2 \\
\theta_3
\end{pmatrix}
=\frac 12
\begin{pmatrix}
rz_3/\lambda-\lambda(x_3x_6-x_4x_5) \\
-\lambda(rx_5-x_3z_4+x_4z_3) \\
\lambda(rx_3+z_4x_5-z_3x_6)
\end{pmatrix}.
\end{equation*}
The corresponding foliation $\F$ is Riemannian but not totally geodesic
for any values of $\lambda\in\rn^+$.
\end{theorem}

\begin{proof}
It is easily checked that the Jacobi relations for $\g$ are equivalent
to the following system of equations
$$\Omega(\lambda,x,y,z)=
\begin{pmatrix}
x_1 & y_3 & y_1-\lambda^2x_3 & z_1-\lambda^2x_5 & y_5+z_3\\
x_2 & y_4 & y_2-\lambda^2x_4 & z_2-\lambda^2x_6 & y_6+z_4
\end{pmatrix}=0.
$$
This implies that the Lie bracket relations take the stated form.
It follows from Proposition \ref{prop-geometry} that the foliation
$\F$ is Riemannian.  The same result shows that $\F$ is not totally
geodesic since $\lambda^2+1=0$ has no real solution.
\end{proof}

\section{The derived algebra $[\k,\k]$ is 2-dimensional}

Let $G$ be a 5-dimensional Riemannian Lie group carrying a left-invariant
minimal and conformal integrable distribution $\V\cong\k$.  Let $\{X,Y,Z,A,B\}$
and $\{X,Y,Z\}$ be orthonormal bases for the Lie algebras $\g$ and $\k$,
respectively, such that
$$[Z,X]=\alpha X+\beta Y\ \ \text{and}\ \ [Z,Y]=\gamma X+\delta Y,$$
where the real structure constants satisfy $\alpha \delta - \beta\gamma\neq 0$.
Then the derived algebra $[\k,\k]$ is 2-dimensional and the additional bracket
relations take the form
\begin{eqnarray*}
\lb XA &=& x_1 X+y_1 Y+z_1 Z,\\
\lb XB &=& x_2 X+y_2 Y+z_2 Z,\\
\lb YA &=& x_3 X+y_3 Y+z_3 Z,\\
\lb YB &=& x_4 X+y_4 Y+z_4 Z,\\
\lb ZA &=&  a_3 A+b_3 B+x_5 X+y_5 Y-(x_1+y_3) Z,\\
\lb ZB &=& -b_3 A+a_3 B+x_6 X+y_6 Y-(x_2+y_4) Z,\\
\lb AB &=& r A+\theta_1 X+\theta_2 Y+\theta_3 Z.
\end{eqnarray*}
As a direct consequence of the two Jacobi equations
$$[[X,Y],A]+[[A,X],Y]+[[Y,A],X]=0$$
$$[[X,Y],B]+[[B,X],Y]+[[Y,B],X]=0$$
we see that, for this particular situation, we have $$z_1=z_2=z_3=z_4=0.$$

\section{The Lie algebras $\g_7(\alpha)$ of Type VII}

For $\alpha\in\rn$ we equip the 3-dimensional Lie algebra $\k=\g_7(\alpha)$
with its standard metric.  Then the bracket relations are given by
$$[Z,X]=\alpha X-Y,\quad [Z,Y]=X+\alpha Y,$$ where $\{X,Y,Z\}$ is
an orthonormal basis for $\k$.  Here the derived algebra
$[\k,\k]$ is 2-dimensional so the bracket relations for $\g$ simplify as
mentioned above.

\begin{theorem}\label{theo-g_7}
Let $G$ be a 5-dimensional Riemannian Lie group carrying a left-invariant,
minimal and conformal distribution $\V$, generated by the Riemannian
subalgebra $\g_7(\alpha)$ as above.  Then the Lie algebra $\g$ of G is
one of the following four given in Examples \ref{exam:7.1} - \ref{exam:7.4}.
\end{theorem}

The reader should note that the following Lie algebras are all
solvable. Further, in the generic case, they are neither nilpotent
nor are the corresponding foliations $\F$ totally geodesic.

\begin{example}\label{exam:7.1}
$$[Z,X]=\alpha X-Y,\ \ [Z,Y]=X+\alpha Y,$$
$$\lb ZA= a_3 A+b_3 B+x_5 X+y_5 Y,$$
$$\lb ZB= -b_3 A+a_3 B+x_6 X+y_6 Y.$$
\end{example}

\begin{example}
$$[Z,X]=\alpha X-Y,\ \ [Z,Y]=X+\alpha Y,$$
$$\lb XA= y_1 Y,\ \ \lb XB= y_2 Y,$$
$$\lb YA= -y_1 X,\ \ \lb YB= -y_2 X,$$
$$\lb ZA= x_5 X+y_5 Y,\ \ \lb ZB= x_6 X+y_6 Y,$$
$$\lb AB= \theta_1 X+\theta_2 Y$$
with
$$(1+\alpha^2)\theta_1={\alpha y_1y_6-\alpha y_2y_5+y_1x_6-x_5y_2},$$
$$(1+\alpha^2)\theta_2={\alpha x_5y_2-\alpha y_1x_6+y_1y_6-y_2y_5}.$$
\end{example}

\begin{example}
$$[Z,X]=\alpha X-Y,\ \ [Z,Y]=X+\alpha Y,$$
$$\lb XB=y_2 Y,\ \ \lb YB= -y_2 X,$$
$$\lb ZA=x_5 X+y_5 Y,\ \ \lb ZB= x_6 X+y_6 Y,$$
$$\lb AB=r A+\theta_1 X+\theta_2 Y$$
where
$$(1+\alpha^2)\theta_1 = -\alpha rx_5-\alpha y_2y_5+ry_5-x_5y_2,$$
$$(1+\alpha^2)\theta_2 = -\alpha ry_5+\alpha x_5y_2-rx_5-y_2y_5.$$
\end{example}

\begin{example}\label{exam:7.4}
$$[Z,X]=\alpha X-Y,\ \ [Z,Y]=X+\alpha Y,$$
$$\lb XB =-y_2(\alpha X-Y),\ \ \lb YB= -y_2(X+\alpha Y),$$
$$\lb ZA=-2\alpha A+x_5 X+y_5 Y,$$
$$\lb ZB=-2\alpha B+x_6 X+y_6 Y+2\alpha y_2 Z,$$
$$\lb AB=2\alpha y_2 A-x_5y_2 X-y_2y_5 Y.$$
\end{example}

\section{The Lie algebras $\mathfrak{sol}_\alpha^3$ of Type VI}

For $\alpha\in\rn^+$ we equip the 3-dimensional Lie algebra
$\k=\mathfrak{sol}_\alpha^3$ with its standard metric. Then the
bracket relations are given by $$[Z,X]=\alpha X,\ \ [Z,Y]=-Y,$$
where $\{X,Y,Z\}$ is an orthonormal basis for $\k$.

\begin{theorem}\label{theo-sol}
Let $G$ be a 5-dimensional Riemannian Lie group carrying a left-invariant,
minimal and conformal distribution $\V$, generated by the Riemannian
subalgebra $\mathfrak{sol}_\alpha^3$ as above.  Then the Lie algebra $\g$ of G is
one of the following ten given in Examples \ref{exam:8.1} - \ref{exam:8.10}.
\end{theorem}

The following Lie algebras are all solvable. Further, in the generic case,
they are neither nilpotent nor are the corresponding foliations $\F$ totally
geodesic.  In two of the cases the foliation is Riemannian.

\begin{example}\label{exam:8.1}
$$[Z,X]=\alpha X,\ \ [Z,Y]=-Y,$$
$$\lb XA = -y_3 X,\ \ \lb XB = -y_4 X,$$
$$\lb YA = y_3 Y,\ \ \lb YB = y_4 Y,$$
$$\lb ZA = x_5 X+y_5 Y,\ \ \lb ZB = x_6 X+y_6 Y,$$
$$\alpha\lb AB = (y_3x_6-y_4x_5) X+\alpha(y_3y_6-y_4y_5) Y,$$
\end{example}

\begin{example}
$$[Z,X]=\alpha X,\ \ [Z,Y]=-Y,$$
$$\lb XB = -y_4 X,\ \ \lb YB = y_4 Y,$$
$$\lb ZA = x_5 X+y_5 Y,\ \ \lb ZB = x_6 X+y_6 Y,$$
$$\alpha\lb AB = \alpha r A -x_5(r+y_4) X+\alpha y_5(r-y_4) Y,$$
\end{example}

\begin{example}
$$[Z,X]=\alpha X,\ \ [Z,Y]=-Y,$$
$${(1-\alpha)}\lb XB = {r\alpha} X,\ \ {(1-\alpha)}\lb YB = -{r} Y,$$
$$\lb ZA = (1-\alpha) A+x_5 X+y_5 Y,$$
$$\lb ZB = (1-\alpha) B+x_6 X+y_6 Y+r Z,$$
$$(1-\alpha)\lb AB = (1-\alpha)r A + rx_5 X + ry_5 Y.$$
\end{example}

\begin{example}
$$[Z,X]=\alpha X,\ \ [Z,Y]=-Y,$$
$$\lb YA = x_3 X,\ \ \lb YB = x_4 X,$$
$$\lb ZA = (1+\alpha) A+x_5 X+y_5 Y,$$
$$\lb ZB = (1+\alpha) B+x_6 X+y_6 Y,$$
$$(2+\alpha)\lb AB = (x_3y_6-x_4y_5) X.$$
\end{example}

\begin{example}
$$[Z,X]=\alpha X,\ \ [Z,Y]=-Y,$$
$$\lb XA = y_1 Y,\ \ \lb XB = y_2 Y,$$
$$\lb ZA = -(1+\alpha) A+x_5 X+y_5 Y,$$
$$\lb ZB = -(1+\alpha) B+x_6 X+y_6 Y,$$
$$(1+2\alpha)\lb AB = (y_2x_5-y_1x_6) Y.$$
\end{example}

\begin{example}
$$[Z,X]=\alpha X,\ \ [Z,Y]=-Y,$$
$$\lb YB = x_4 X,$$
$$\lb ZA = (1+\alpha) A+x_5 X+y_5 Y,$$
$$\lb ZB = (1+\alpha) B+x_6 X+y_6 Y,$$
$$(2+\alpha)\lb AB = -x_4y_5 X.$$
\end{example}

\begin{example}
$$[Z,X]=\alpha X,\ \ [Z,Y]=-Y,$$
$$\lb XB = y_2 Y,$$
$$\lb ZA = -(1+\alpha) A+x_5 X+y_5 Y,$$
$$\lb ZB = -(1+\alpha) B+x_6 X+y_6 Y,$$
$$(1+2\alpha)\lb AB = y_2x_5 Y.$$
\end{example}

\begin{example}
$$[Z,X]=\alpha X,\ \ [Z,Y]=-Y,$$
$$\lb ZA =  a_3 A+b_3 B+x_5 X+y_5 Y,$$
$$\lb ZB = -b_3 A+a_3 B+x_6 X+y_6 Y.$$
\end{example}

\begin{example}
$$[Z,X]=\alpha X,\ \ [Z,Y]=-Y,$$
$$2\lb ZA = -A+2b_3B+2x_5 X+2y_5 Y,$$
$$2\lb ZB = -2b_3A-B+2x_6 X+2y_6 Y,$$
$$\lb AB = \theta_2 Y.$$
\end{example}

\begin{example}\label{exam:8.10}
$$[Z,X]=\alpha X,\ \ [Z,Y]=-Y,$$
$$2\lb ZA =  \alpha A+2b_3 B+2x_5 X+2y_5 Y,$$
$$2\lb ZB = -2b_3 A+\alpha B+2x_6 X+2y_6 Y,$$
$$\lb AB = \theta_1 X.$$
\end{example}

\section{The Lie algebra  $\h^3$  of  Type V}

Let the Lie algebra $\k=\h^3$ be equipped with its standard metric.
Then its bracket relations are given by $$[Z,X]= X,\ \ [Z,Y]= Y,$$
where $\{X,Y,Z\}$ is an orthonormal basis for $\k$.

\begin{theorem}\label{theo-h3}
Let $G$ be a 5-dimensional Riemannian Lie group carrying a left-invariant,
minimal and conformal distribution $\V$, generated by the Riemannian
subalgebra $\h^3$ as above.  Then the Lie algebra $\g$ of G is
one of the following nine given in Examples \ref{exam:9.1} - \ref{exam:9.9}.
\end{theorem}

The reader should note that the following Lie algebras are all solvable.
Further, in the generic case, they are neither nilpotent nor are the
corresponding foliations $\F$ totally geodesic.

\begin{example}\label{exam:9.1}
$$[Z,X]=X,\ \ [Z,Y]=Y,$$
$$\lb XA = -y_3 X+y_1 Y,\
\ 2y_3\lb XB = -2y_3y_4 X+y_1(2y_4-r) Y,$$
$$y_1\lb YA = -y_3^2 X+y_3 Y,\
\ 2y_1\lb YB = -y_3(2y_4+r) X+2y_1y_4 Y,$$
$$\lb ZA = x_5 X+y_5 Y,$$
$$\lb ZB = x_6 X+y_6 Y,$$
$$\lb AB = r A+\theta_1 X+\theta_2 Y.$$
$$-2y_1\theta_1=2rx_5y_1+ry_3y_5+2x_5y_1y_4-2x_6y_1y_3-2y_3^2y_6+2y_3y_4y_5$$
$$-2y_3\theta_2=rx_5y_1+2ry_3y_5-2x_5y_1y_4+2x_6y_1y_3+2y_3^2y_6-2y_3y_4y_5$$
\end{example}

\begin{example}
$$[Z,X]= X,\ \ [Z,Y]= Y,$$
$$\lb XA = -y_3 X+y_1 Y,\ \ y_3\lb XB = -y_3y_4 X+y_1y_4 Y,$$
$$\lb YA =  x_3 X+y_3 Y,\ \ y_3\lb YB =  x_3y_4 X+y_3y_4 Y,$$
$$\lb ZA = x_5 X+y_5 Y,\ \ \lb ZB = x_6 X+y_6 Y,$$
$$\lb AB = \theta_1 X+\theta_2 Y,$$
$$\theta_1y_3 = x_3y_4y_5-x_3y_3y_6-x_5y_3y_4+x_6y_3^2,$$
$$\theta_2y_3 = x_5y_1y_4-x_6y_1y_3+y_3y_4y_5-y_6y_3^2.$$
\end{example}

\begin{example}
$$[Z,X]= X,\ \ [Z,Y]= Y,$$
$$\lb XB =-y_4 X+y_2 Y,\ \ \lb YB = x_4 X+y_4 Y,$$
$$\lb ZA = x_5 X+y_5 Y,\ \ \lb ZB = x_6 X+y_6 Y,$$
$$\lb AB = r A-(rx_5-x_4y_5+x_5y_4) X-(ry_5-x_5y_2-y_4y_5) Y.$$
\end{example}

\begin{example}
$$[Z,X]= X,\ \ [Z,Y]= Y,$$
$$2\lb ZA = A+ 2b_3 B+2x_5 X+2y_5 Y,$$
$$2\lb ZB = -2b_3 A+ B+2x_6 X+2y_6 Y,$$
$$\lb AB = \theta_1 X+\theta_2 Y.$$
\end{example}

\begin{example}
$$[Z,X]= X,\ \ [Z,Y]= Y,$$
$$2\lb XB = r X,$$
$$\lb YA = x_3 X,\ \ 2\lb YB = 2x_4 X- r Y,$$
$$\lb ZA = x_5 X+y_5 Y,\ \ \lb ZB = x_6 X+y_6 Y,$$
$$2\lb AB = 2r A+ (2x_4y_5-2y_6x_3-rx_5) X-3ry_5 Y.$$
\end{example}

\begin{example}
$$[Z,X]= X,\ \ [Z,Y]= Y,$$
$$\lb XA = y_1 Y,\ \ 2\lb XB = -r X+2y_2 Y,$$
$$2\lb YB = r Y,$$
$$\lb ZA = x_5 X+y_5 Y,\ \ \lb ZB = x_6 X+y_6 Y,$$
$$2\lb AB = 2r A - 3rx_5 X + (2x_5y_2-2x_6y_1-ry_5) Y.$$
\end{example}

\begin{example}
$$[Z,X]= X,\ \ [Z,Y]= Y,$$
$$x_4\lb XA = x_3y_2 Y,\ \ \lb XB = y_2 Y,$$
$$\lb YA = x_3 X,\ \ \lb YB = x_4 X,$$
$$\lb ZA = x_5 X+y_5 Y,\ \ \lb ZB = x_6 X+y_6 Y,$$
$$\lb AB = (x_4y_5-x_3y_6) X + (x_5y_2-x_6y_1) Y.$$
\end{example}

\begin{example}
$$[Z,X]= X,\ \ [Z,Y]= Y,$$
$$\lb ZA =  a_3 A+b_3 B+x_5 X+y_5 Y,$$
$$\lb ZB = -b_3 A+a_3 B+x_6 X+y_6 Y,$$
\end{example}

\begin{example}\label{exam:9.9}
$$[Z,X]= X,\ \ [Z,Y]= Y,$$
$$2\lb XB = -r X,\ \ 2\lb YB = -r Y,$$
$$\lb ZA = -2 A+x_5 X+y_5 Y,$$
$$\lb ZB = -2 B+x_6 X+y_6 Y+r Z,$$
$$2\lb AB = 2r A-rx_5 X-ry_5 Y.$$
\end{example}

\section{The Lie algebra  $\g_4$  of  Type IV}

We equip the Lie algebra $\k=\g_4$ with its standard metric.
Then its bracket relations are given by $$[Z,X]= X,\ \ [Z,Y]=X+Y,$$
where $\{X,Y,Z\}$ is an orthonormal basis for $\k$.

\begin{theorem}\label{theo-h3}
Let $G$ be a 5-dimensional Riemannian Lie group carrying a left-invariant,
minimal and conformal distribution $\V$, generated by the Riemannian
subalgebra $\g_4$ as above.  Then the Lie algebra $\g$ of G is
one of the following five given in Examples \ref{exam:10.1} - \ref{exam:10.5}.
\end{theorem}

Note that the following Lie algebras are all solvable.
Further, in the generic case, they are neither nilpotent
nor are the corresponding foliations $\F$ totally geodesic.
In two of the cases the foliation is Riemannian.

\begin{example}\label{exam:10.1}
$$[Z,X]= X,\ \ [Z,Y]=X+Y,$$
$$\lb ZA= a_3 A+b_3 B+x_5 X+y_5 Y,$$
$$\lb ZB=-b_3 A+a_3 B+x_6 X+y_6 Y.$$
\end{example}

\begin{example}
$$[Z,X]= X,\ \ [Z,Y]=X+Y,$$
$$2\lb ZA =A+2b_3 B+2x_5 X+2y_5 Y,$$
$$2\lb ZB =-2b_3 A+ B+2x_6 X+2y_6 Y,$$
$$\lb AB =\theta_1 X.$$
\end{example}

\begin{example}
$$[Z,X]= X,\ \ [Z,Y]=X+Y,\ \ \lb YB=x_4 X,$$
$$\lb ZA=x_5 X+y_5 Y,\ \ \lb ZB=x_6 X+y_6 Y,$$
$$\lb AB= r A-(rx_5-ry_5-x_4y_5) X-ry_5 Y.$$
\end{example}

\begin{example}
$$[Z,X]= X,\ \ [Z,Y]=X+Y,$$
$$\lb YA=x_3 X,\ \ \lb YB=x_4 X,$$
$$\lb ZA=x_5 X+y_5 Y,\ \ \lb ZB=x_6 X+y_6 Y,$$
$$\lb AB=(x_4y_5-x_3y_6) X.$$
\end{example}

\begin{example}\label{exam:10.5}
$$[Z,X]= X,\ \ [Z,Y]=X+Y,$$
$$\lb XB=y_4 X,\ \ \lb YB=y_4 X+y_4 Y,$$
$$\lb ZA=-2 A+x_5 X+y_5 Y,$$
$$\lb ZB=-2 B+x_6 X+y_6 Y-2y_4 Z,$$
$$\lb AB=-2y_4 A+x_5y_4 X+y_4y_5 Y.$$
\end{example}

\vskip .5cm

\begin{remark}
The number of solutions grows rapidly with decreased dimension
of the derived algebra $[\k,\k]$.  When this dimension is either
0 or 1, as for types I-III, the solutions are so many that it is
impossible to list them all here in this short paper.
\end{remark}

\appendix

\section{The $3$-dimensional Lie algebras}\label{section-Bianchi}

At the end of the 19th century, L.~Bianchi classified the $3$-dimensional
real Lie algebras.  They fall into nine disjoint types I-IX.
Each contains a single isomorphy class except types VI and VII which
are continuous families of different classes.

\begin{example}[I]
By $\rn^3$ we denote both the 3-dimensional abelian Lie algebra
and the corresponding simply connected Lie group.
\end{example}

\begin{example}[II]
The nilpotent Heisenberg algebra $\mathfrak{nil}^3$ is defined by
$$[Z,Y]=X.$$ The corresponding simply connected Lie group is
$\text{Nil}^3$.
\end{example}

\begin{example}[III]
The solvable Lie algebra $\h^2\oplus\rn$ is given by $$[Z,X]=X.$$
The corresponding simply connected Lie group is the product
$H^2\times\rn$ of the real line and the group $H^2$ classically
modelling the standard hyperbolic plane.
\end{example}

\begin{example}[IV]
By $\g_4$ we denote the solvable Lie algebra given by
$$[Z,X]=X,\quad [Z,Y]=X+Y.$$
The corresponding simply connected Lie group is $G_4$.
\end{example}

\begin{example}[V]
The solvable Lie algebra $\h^3$ is defined by $$[Z,X]=X,\quad [Z,Y]=Y.$$
The corresponding simply connected group $H^3$ classically models the
standard hyperbolic 3-space.
\end{example}

\begin{example}[VI]
For $\alpha\in\rn^+$, the solvable Lie algebra $\sol_\alpha^3$ is given by
$$[Z,X]=\alpha X,\quad [Z,Y]=-Y.$$
The corresponding simply connected Lie group is denoted by
$\text{Sol}_\alpha^3$.
\end{example}

\begin{example}[VII]
For $\alpha\in\rn$, the solvable Lie algebra $\g_7(\alpha)$ is defined
by $$[Z,X]=\alpha X-Y,\quad [Z,Y]=X+\alpha Y.$$ We denote the
corresponding simply connected Lie group by $G_7(\alpha)$.
\end{example}

\begin{example}[VIII]
The simple Lie algebra $\slr 2$ satisfies
$$[X,Y]=2Z,\quad [Z,X]=2Y,\quad [Y,Z]=-2X.$$
The corresponding simply connected Lie group is denoted by
$\widetilde{\SLR 2}$ as it is the universal cover of the special
linear group $\SLR 2$.
\end{example}

\begin{example}[IX]
The simple Lie algebra $\su 2$ satisfies
$$[X,Y]=2Z,\quad [Z,X]=2Y,\quad [Y,Z]=2X.$$
The corresponding simply connected Lie group is $\SU 2$
diffeomorphic to the standard 3-dimensional sphere.
\end{example}

\section{Acknowledgements}
The author is grateful to the referee for useful comments on this paper.

\end{document}